\documentclass{amsart}

\usepackage{hyperref}
\usepackage{amsmath, amsthm}%
\usepackage{amsfonts}%
\usepackage{amssymb}%
\usepackage{graphicx}
\usepackage{amscd, amstext}
\usepackage{pictexwd}
\usepackage{dcpic}
\usepackage[mathscr]{eucal}
\usepackage{multicol}
\usepackage{multirow}
\usepackage{color}
\usepackage{bm}
\usepackage{makecell} 

\theoremstyle{plain}
\newtheorem{theorem}{Theorem}[section]
\newtheorem{lemma}[theorem]{Lemma}
\newtheorem{proposition}[theorem]{Proposition}
\newtheorem{corollary}[theorem]{Corollary}
\newtheorem{algorithm}[theorem]{Algorithm}
\newtheorem*{maintheorem*}{Main Theorem}
\newtheorem*{corollary*}{Corollary}
\newtheorem*{lemma*}{Lemma}
\newtheorem*{keylemma*}{Key Lemma}

\theoremstyle{definition}
\newtheorem{definition}[theorem]{Definition}
\newtheorem{example}[theorem]{Example}

\theoremstyle{remark}
\newtheorem{remark}[theorem]{Remark}

\numberwithin{equation}{section}

\newcommand{\etalchar}[1]{$^{#1}$}

\newcommand{\F}{\mathbb{F}}

\newcommand{\PP}{\mathbb{P}}

\begin{document}

\title{Hasse-Witt matrices and mirror toric pencils}
\author{Adriana Salerno and Ursula Whitcher}

\begin{abstract}
	Mirror symmetry suggests unexpected relationships between arithmetic properties of distinct families of algebraic varieties. For example, Wan and others have shown that for some mirror pairs, the number of rational points over a finite field matches modulo the order of the field. In this paper, we obtain a similar result for certain mirror pairs of toric hypersurfaces. We use recent results describing the relationship between the Picard-Fuchs equations of these varieties and their Hasse--Witt matrices, which encapsulate information about the number of points, to compute the number of points modulo the order of the field explicitly. We classify pencils of K3 hypersurfaces in Gorenstein Fano toric varieties where the point count coincides and analyze examples related to classical hypergeometric functions.
\end{abstract}

\maketitle

The authors thank the anonymous referee and the editor of ATMP for very helpful comments. They are indebted to Michael Allen, Philip Candelas, Robert Davis, Tyler Kelly, Andrey Novoseltsev, and Xenia de la Ossa for interesting and inspiring conversations. They thank the Isaac Newton Institute for Mathematical Sciences, Cambridge, for support and hospitality during the K-theory, algebraic cycles and motivic homotopy theory programme, where work on this paper was undertaken. This
	work was supported by EPSRC grant no EP/R014604/1. For Adriana Salerno, this material is based upon work supported by and while
	serving at the National Science Foundation. Any opinion, findings, and conclusions
	or recommendations expressed in this material are those of the authors and do not
	necessarily reflect the views of the National Science Foundation.

\section{Introduction}

Let $\F_q$ be a finite field of order $q = p^a$. Then the Frobenius operator $\mathcal{F}$ induces a $p$-linear map

\[H^n(X, \mathcal{O}_X) \xrightarrow{\mathcal{F}} H^n(X, \mathcal{O}_X).\]

\noindent A matrix for $\mathcal{F}$, for some choice of basis of $H^n(X, \mathcal{O}_X)$, is called a \emph{Hasse--Witt matrix}. Katz made a detailed study of the Hasse--Witt matrix for projective hypersurfaces (see for example \cite{katz} and \cite{katzCongruence}); Achter and Howe give a useful description of the Hasse--Witt matrix using modern language in \cite{AH}. When $H^n(X, \mathcal{O}_X)$ is one-dimensional, the Hasse--Witt matrix is simply an element of $\F_q$. Calabi--Yau varieties have this property, as do elliptic curves and K3 surfaces, their lower-dimensional counterparts. 

The Hasse--Witt matrix encapsulates information about the number of points on a Calabi--Yau variety over a finite field of characteristic $p$, modulo $p$. The relationship is given by
Katz's congruence formula (\cite[Th\'eor\`eme 3.1]{katzCongruence}), which relates the generating function $Z(X/\F_q;T)$ for the number of points on $X$ over finite extensions of $\F_q$ to the Frobenius action:
\begin{equation}
	Z(X/\F_q;T) \equiv \prod_{i=0}^n \det(1-T \cdot \mathcal{F}^a|H^i(X, \mathcal{O}_X ))^{(-1)^{i+1}} \pmod{p}.
\end{equation}
\noindent The zeta function of a Calabi--Yau variety has a unique root that is a $p$-adic unit, known as the \emph{unit root}.  Combining this fact with Katz's congruence formula, we see that the number of points on a Calabi--Yau variety (or its lower-dimensional counterpart) $\pmod{p}$ is entirely controlled by the factor of the zeta function given by the Hasse--Witt matrix.

We are interested in the interplay between the geometric and arithmetic properties of K3 surfaces and Calabi--Yau varieties. We will work with varieties realized as hypersurfaces in Gorenstein Fano toric varieties. In general, such hypersurfaces arise in multiparameter families. Their holomorphic periods satisfy systems of partial differential equations that have an $A$-hypergeometric structure (these are also known as GKZ equations, after {Gel${}^\prime$fand}, Kapranov, and Zelevinsky). However, in many cases one can specialize to subfamilies of independent geometric interest. We will use combinatorial and geometric criteria to identify collections of pencils of varieties with similar, but not identical, arithmetic properties. 

In the case of K3 surfaces, we identify pencils that correspond to classical hypergeometric equations of the form ${}_3F_2$. These are lower-dimensional analogues of the famous 14~cases of ${}_4F_3$ hypergeometric equations corresponding to Calabi-Yau hypersurfaces with maximally unipotent monodromy (see the work of \cite{RV, DM, LTYZ}). One may use this fact to obtain explicit formulas for the Hasse--Witt invariants.

Our choice of geometric context is inspired by mirror symmetry, which describes deep and unexpected connections between distinct families of algebraic varieties. The first mirror construction, due to Greene and Plesser in \cite{GP}, constructs a mirror to the family of all smooth Calabi--Yau hypersurfaces in $\PP^n$, using the Fermat pencil of Calabi-Yau $n-1$-folds $X_{n-1,\psi}$ given by

\[x_0^{n+1} + \dots + x_n^{n+1} - \psi (n+1) x_0 \dots x_n = 0.\]

Candelas, de la Ossa, and Rodr\'{i}guez Villegas observed that in the case of Calabi-Yau threefolds, the Greene--Plesser mirror construction has arithmetic consequences. In this case, a group $G \cong (\mathbb{Z}/5\mathbb{Z})^3$ acts diagonally on $X_{3,\psi}$, and the mirror to smooth quintics in $\PP^4$ is given by the resolution of singularities $Y_{3,\psi} = \widetilde{X_{3,\psi}/G}$. In \cite{CORV}, they showed that for $\psi \in \mathbb{Z}$ the number of points on $X_{3,\psi}$ over a field of prime order is given by truncations of series solutions of the Picard--Fuchs differential equation satisfied by the holomorphic form. This phenomenon generalizes results of Dwork in \cite{padic} for the case of quartics $X_{2,\psi}$. In \cite{CORV2}, Candelas, de la Ossa, and Rodr\'{i}guez Villegas computed the zeta functions of $X_{3,\psi}$ and the mirror quintic pencil $Y_{3,\psi}$. They showed that for each $\psi \in \mathbb{Z}$ the zeta functions share a common factor corresponding to the holomorphic form. 

Various researchers in mirror symmetry have sought to generalize the phenomena observed by Candelas--de la Ossa--Rodr\'{i}guez Villegas to broader classes of examples of mirror varieties. The first work built on the Greene--Plesser mirror construction. In \cite{wan}, Wan showed that for any $\psi \in \F_q$ and any dimension $n$, the number of points on $X_{n-1,\psi}$ and its mirror $Y_{n-1,\psi}$ over a field of $q^k$ elements are the same, modulo the order of the field:
\[\#(X_{n-1,\psi}, \F_{q^k}) \equiv \#(Y_{n-1,\psi}, \F_{q^k}) \pmod{q^k}.\]
Wan conjectured that a similar congruence should hold in any setting where one can construct a one-to-one correspondence between a Calabi-Yau variety $X$ and its mirror $Y$ (rather than having more general paired mirror families, as in the correspondence between all smooth quintics in $\PP^4$ and the quintic mirror $Y_{3,\psi}$). In \cite{kadir2}, Kadir considered a two-parameter family of octic Calabi-Yau threefolds in the weighted projective space $\PP(1,1,2,2,2)$. She showed that this family and the family of mirror octics obtained via a Greene--Plesser mirror construction share a common factor in their zeta functions corresponding to the holomorphic form, and claimed that a similar computation would describe a common factor for any Greene-Plesser mirror of Calabi-Yau hypersurfaces in a Gorenstein weighted projective space. In the arithmetically rich K3 surface setting, researchers have used the Greene--Plesser mirror construction to investigate the relationship between K3 surface families of high Picard rank and modularity (see \cite{LY, VY, Doran}).

Aldi and Peruni\v{c}i\'c in \cite{AP}, and the present authors together with Doran, Kelly, Sperber, and Voight in \cite{zeta} and \cite{hypergeometric}, have investigated arithmetic mirror symmetry phenomena in the context of the Berglund-H\"{u}bsch-Krawitz (BHK) mirror construction. In particular, \cite{zeta} uses BHK mirrors to identify pencils of Calabi-Yau varieties in $\mathbb{P}^n$ that share common Picard--Fuchs equations and common factors in their zeta functions. Kloosterman further explored this phenomenon in \cite{kloosterman}, using a geometric approach. Both \cite{kadir2} and \cite{CORV2} use the Batyrev mirror construction and techniques of toric varieties for a more detailed analysis of the Greene--Plesser mirror. In \cite{MW}, Magyar and the second author examined a larger collection of Batyrev mirrors, providing a conjectural description of collections of toric hypersurface pencils that are strong mirrors in Wan's sense. The recent works \cite{BKSZ, COES, COS} use the expected correspondence between Picard-Fuchs equations and zeta functions together with novel computational techniques to identify persistent zeta function factorizations of both physical and arithmetic significance.

All of these variations on arithmetic mirror symmetry compare properties of special pencils or subfamilies where the Picard--Fuchs equation satisfied by the holomorphic form is the same on each side of the mirror correspondence. From the standpoint of mirror constructions, this is a very special property: one typically expects that a family with many complex deformations, such as all smooth quintics in $\PP^4$, should be mirror to a family with few complex deformations, such as the Greene--Plesser mirror $Y_{3,\psi}$. 

From the standpoint of arithmetic, on the other hand, this restriction is straightforward. Point counts on a variety over a finite field may be computed using the Frobenius action $\mathcal{F}$ on an appropriately chosen $p$-adic cohomology, so the question is whether the Frobenius action and the cohomological structure measured by the Picard--Fuchs equation are compatible. In the arithmetic setting, such investigations go back to Igusa's classic study of the Legendre pencil of elliptic curves \cite{igusa} and have been worked out in detail in various contexts, building on Dwork's work in \cite{padic}. Slinkin and Varchenko drew on these ideas to give arithmetic versions of the solutions to the KZ equations used in conformal field theory in \cite{slinkinvarchenko}. The analysis in \cite{zeta} draws heavily on the arithmetic tradition in order to make the correspondence between Picard--Fuchs equations and point counting precise in the setting of BHK mirror symmetry.

The mirror-symmetric context of the present work is Batyrev mirror symmetry, which describes mirrors of Calabi--Yau varieties realized as hypersurfaces in Gorenstein Fano toric varieties. Such toric varieties correspond to combinatorial objects called reflexive polytopes; these polytopes have been completely classified in dimensions four and lower. The reflexive simplices, $n$-dimensional reflexive polytopes with $n+1$ vertices, determine Gorenstein weighted projective spaces and certain finite quotients. In this particular combinatorial setting, the Batyrev, Greene--Plesser, and BHK constructions overlap. Our strategy is to generalize techniques appropriate to this overlap to a broader class of examples.

The authors of \cite{HLYY} describe a relationship between Picard--Fuchs equations and the Hasse--Witt matrices of Calabi-Yau varieties realized as toric hypersurfaces. Their results extend results of Adolphson and Sperber in \cite{AS} for Calabi-Yau hypersurfaces in projective space and prove a conjecture made by Vlasenko in \cite{vlasenko}. A more arithmetic and algorithmic perspective on Hasse--Witt matrices and associated cohomology theories, including an alternate proof of the conjecture, may be found in \cite{BV}.

In this paper, we apply the results of \cite{HLYY} to characterize arithmetic mirror symmetry phenomena for certain pencils of elliptic curves, K3 surfaces, or Calabi--Yau hypersurfaces in toric varieties. We thereby prove the congruence for toric diagonal pencils conjectured in \cite{MW} and obtain explicit examples of the relationships between periods and point counts described in \cite{HLYY} and \cite{BV}.

In more detail, given a reflexive polytope $\Delta$, we use the vertices of the polar polytope $\Delta^\circ$ to define a pencil of Calabi--Yau varieties called the \emph{vertex pencil} (see Definition~\ref{D:vertexPencil}). Recall that two polytopes are \emph{combinatorially equivalent} if there is a bijection between their faces that preserves inclusions. For example, any two simplices of the same dimension are combinatorially equivalent. We will use a stronger condition on polytopes (see also Definition~\ref{D:kernelpair}): we say two combinatorially equivalent polytopes are a \emph{kernel pair} if the matrices given by their vertices have the same kernel. Using this notion, we characterize combinatorially equivalent reflexive polytopes whose vertex pencils share arithmetic properties. If a pair of polytopes are both polar dual and a kernel pair we say they are a \emph{mirror kernel pair}.

\begin{keylemma*}[see Lemma~\ref{T:mainProof}]\label{T:main}
Let $\Delta$ and $\Gamma$ be a kernel pair of $n$-dimensional reflexive polytopes, and suppose $\Delta^\circ$ and $\Gamma^\circ$ are also a kernel pair. Let $V_\Delta$ and $V_\Gamma$ be smooth toric varieties determined by maximal simplicial refinements of the fans over the faces of $\Delta$ and $\Gamma$, respectively. Let $X_{\Delta,\psi}$ and $X_{\Gamma, \psi}$ be the corresponding vertex pencils. Then for any rational $\psi$ and prime $p$ such that $X_{\Delta,\psi}$ and $X_{\Gamma, \psi}$ are smooth, their Hasse--Witt matrices are the same.
\end{keylemma*}

The Key Lemma immediately implies a relationship between point counts:

\begin{corollary*}[See Corollary~\ref{C:pointCountsModp}]
Let $\Delta$ and $\Gamma$ be a kernel pair of $n$-dimensional reflexive polytopes, and suppose $\Delta^\circ$ and $\Gamma^\circ$ are also a kernel pair. Let $V_\Delta$ and $V_\Gamma$ be smooth toric varieties determined by maximal simplicial refinements of the fans over the faces of $\Delta$ and $\Gamma$, respectively. Let $X_{\Delta,\psi}$ and $X_{\Gamma, \psi}$ be the corresponding vertex pencils. Then for any rational $\psi$ and prime $p$ such that $X_{\Delta,\psi}$ and $X_{\Gamma, \psi}$ are smooth, 
\[\# X_{\Delta,\psi} \equiv \# X_{\Gamma, \psi} \pmod{p}.\]
\end{corollary*}

The Key Lemma provides a combinatorial characterization of toric hypersurfaces that are ``strong mirrors'' in the sense of \cite{wan}. Its proof uses the relationship between Hasse--Witt matrices and Picard--Fuchs equations. 

In \S~\ref{S:toric}, in addition to reviewing necessary background in toric geometry, we use the classification of Gorenstein Fano toric varieties in low dimensions to obtain a complete list of mirror pairs of reflexive polytopes in dimensions 2 and 3 that are also kernel pairs. In 2~dimensions, these polytopes determine Gorenstein Fano toric varieties that are weighted projective spaces, $\mathbb{P}^1 \times \mathbb{P}^1$, or certain finite quotients; in 3~dimensions, the possibilities are weighted projective spaces, finite quotients of weighted projective spaces, and examples studied in \cite{MW}. The corresponding vertex pencils include examples whose modularity properties were studied in \cite{LY} and \cite{VY}. We then prove the Key Lemma in \S~\ref{S:keylemma}.

In \S~\ref{S:examples}, we develop techniques for effective computation of $\mathrm{HW}_p(X_{\Gamma,\psi}) \pmod{p}$ or, equivalently, $\# X_{\Delta,\psi} \equiv \# X_{\Gamma, \psi} \pmod{p}$. We use these techniques to study the arithmetic of vertex pencils associated to kernel pairs. We first discuss a pencil of elliptic curves in $\mathbb{P}^1 \times \mathbb{P}^1$ and its mirror, then study collections of pencils of K3 surfaces associated to kernel pairs. Our main theorem characterizes K3 vertex pencils associated to mirror kernel pairs of polytopes that have hypergeometric structures. This theorem generalizes examples studied by Dwork and Katz to the setting of toric hypersurfaces and produces examples of K3 surface pencils whose moduli are related to the moduli of elliptic curves in interesting ways. 

\begin{maintheorem*}
	There are 32 mirror kernel pairs of three-dimensional reflexive polytopes, of which 6 are self-dual. The mirror kernel pairs are divided into 16 types with a common kernel; each type is associated to vertex pencils with common Picard--Fuchs equations. Four of these types correspond to vertex pencils of K3 surfaces with general Picard number 19 over $\mathbb{C}$. If $\Delta$ is a reflexive polytope of one of these four types, the Hasse--Witt invariant of a member $\# X_{\Delta,\psi}$ of the associated vertex pencil satisfies a hypergeometric truncation relationship for any rational $\psi$, as given in the following table.
	
	    \begin{tabular}{|c|c|c|} %
		\hline
		\textbf{Associated to \dots} & \textbf{Polytope pairs} & \textbf{Hasse--Witt invariant}\\
		\hline
		Fermat quartic in $\mathbb{P}^3$ & \makecell{$(0, 4311)$,\\ $(8, 3313),$ \\ $(427, 427)$,\\ $(429, 429)$} &  $\left[{}_3F_2\left(\frac{1}{2},\frac{1}{4},\frac{3}{4};1,1\mid \frac{256}{\psi^4}\right)\right]^{(p-1)}\bmod p$\\
		 \hline
		Sextic in $\mathbb{P}(1,1,1,3)$ & \makecell{$(2, 4317)$,\\ $(85, 3726),$ \\ $(741, 1943)$} & $\left[{}_3F_2\left(\frac{1}{2},\frac{1}{6},\frac{5}{6};1,1\mid \frac{1728}{\psi^6}\right)\right]^{(p-1)}\bmod p$ \\
		 \hline
		Group I & \makecell{$(3, 4283)$,\\ $(753, 754)$} & $\left[{}_3F_2\left(\frac{1}{2},\frac{1}{3},\frac{2}{3};1,1\mid -\frac{108}{\psi^3}\right)\right]^{(p-1)}\bmod p$ \\
		 \hline
		Group II & \makecell{$(10, 4314)$,\\ $(433, 3316)$,\\ $(436, 3321)$} & $\left[{}_3F_2\left(\frac{1}{2},\frac{1}{4},\frac{3}{4};1,1\mid \frac{256}{\psi^4}\right)\right]^{(p-1)}\bmod p$\\
		 \hline
	\end{tabular}
\end{maintheorem*}

The hypergeometric parameters that we obtain are special; in the language of \cite{BCM}, they are defined over $\mathbb{Q}$ (see Definition~\ref{D:overQ}) and correspond to solutions with maximally unipotent monodromy at the origin. In low dimensions, such parameters were classified in \cite{RV}. For ${}_4F_3$ hypergeometric functions, there are 14 cases. The properties of one-parameter families of Calabi-Yau threefolds whose Picard-Fuchs equations are given by these parameters have been studied by many authors; see, for example, \cite{CYY, DM, 14thcase} and the recent work of \cite{LTYZ}. In the case of K3 surfaces, there are 4 multisets of hypergeometric parameters analogous to the famous 14 cases for Calabi-Yau threefolds. Our main theorem provides a uniform geometric realization for three of these four cases. We show in Section~\ref{S:fourcases} that the Picard-Fuchs equation of a highly symmetric vertex pencil studied in \cite{KLMSW} is hypergeometric with parameters $(\frac{1}{2},\frac{1}{2},\frac{1}{2}; 1,1)$. We thus obtain a geometric realization result as a corollary:

\begin{corollary*}[See Corollary~\ref{C:4multisets}]
	Each of the four multisets of hypergeometric parameters yielding ${}_3F_2$ hypergeometric functions defined over $\mathbb{Q}$ with maximally unipotent monodromy at the origin corresponds to the Picard-Fuchs equation of a vertex pencil of K3 surfaces realized as hypersurfaces in a Gorenstein Fano toric variety.
\end{corollary*}

\section{Toric hypersurfaces and Picard--Fuchs equations}\label{S:toric}

\subsection{Batyrev mirrors and special pencils}

Our study is focused on particular families of hypersurfaces associated to reflexive polytopes in combinatorially natural ways. We begin by reviewing Batyrev mirror symmetry for toric Calabi--Yau hypersurfaces and establishing notation, then give combinatorial and geometric descriptions of the particular objects of interest to us. For more detailed references for Batyrev mirror symmetry, the reader may consult \cite{CK} for a more general expository treatment, or \cite{k3expository} for a discussion focused on K3 hypersurfaces. We take K3 hypersurfaces in the weighted projective space $\mathbb{P}(1,1,1,3)$ and their mirrors as a running example.

Let $N \cong \mathbb{Z}^k$ be a lattice, and let $M \cong \mathrm{Hom}(N, \mathbb{Z})$ be the dual lattice. Write $N_\mathbb{R}$ and $M_\mathbb{R}$ for the associated real vector spaces, and let $\langle \; , \; \rangle$ represent the bilinear pairing on $N_\mathbb{R}$ and $M_\mathbb{R}$ induced by the duality of $N$ and $M$. Let $\Delta$ be a polytope in $N_\mathbb{R}$, and assume that $\Delta$ contains the origin strictly in its interior. The \emph{polar polytope} of $\Delta$ is given by 
\[\Delta^\circ = \{ w \in M_\mathbb{R} \mid \langle v, w \rangle \geq -1 \text{ for all } v \in \Delta\}.\]
\noindent Note that $(\Delta^\circ)^\circ = \Delta$. If $\Delta$ and $\Delta^\circ$ have vertices in the lattices $N$ and $M$ respectively, each polytope is a \emph{reflexive polytope} and we say that $\Delta$ and $\Delta^\circ$ are a \emph{mirror pair} of polytopes.

\begin{example}
	Let $\Delta$ be the simplex with vertices $(1,  0,  0)$, $(0,  1,  0)$, $(0,  0,  1)$, and $(-3, -1, -1)$. The polytope $\Delta$ is reflexive and its polar dual is the simplex $\Delta^\circ$ with vertices $( 1, -1, -1)$, $(-1,  5, -1)$, $(-1, -1,  5)$, and $(-1, -1, -1)$.
\end{example}

We will use a more restrictive combinatorial condition.

\begin{definition}\label{D:kernelpair}
	Let $\Delta$ and $\Gamma$ be combinatorially equivalent $n$-dimensional reflexive polytopes with $k$ vertices. If the kernels of the matrices whose columns are given by the vertices of the polytopes $\Delta$ and $\Gamma$ are the same submodule of $\mathbb{Z}^k$, we say $\Delta$ and $\Gamma$ are a \emph{kernel pair}.
\end{definition}

\begin{example}\label{E:kernelpair}
	Let $\Delta$ be the simplex with vertices $(1,  0,  0)$, $(0,  1,  0)$, $(0,  0,  1)$, and $(-3, -1, -1)$, and let $\Delta^\circ$ be its polar dual. Then $\Delta$ and $\Delta^\circ$ are a kernel pair, with kernel generated by the element $(3,1,1,1)$ of $\mathbb{Z}^4$.
\end{example}

In a case such as Example~\ref{E:kernelpair}, where two reflexive polytopes are both a mirror pair and a kernel pair, we concatenate the adjectives and refer to them as a mirror kernel pair.

We may use a lattice polytope to define a fan in two ways: by taking the fan over the faces of the polytope, or by taking the \emph{normal fan} to the polytope, whose cones consist of the normal cones to each face of the polytope.
If $\Delta$ and $\Delta^\circ$ are a mirror pair, then these notions are dual: the fan over the faces of $\Delta$ is identical to the normal fan of  $\Delta^\circ$. One may also refine the fan over the faces of a polytope using new one-dimensional cones corresponding to the lattice points of a polytope. 

The fan $R$ over the faces of a reflexive polytope $\Delta$ determines a $k$-dimensional Gorenstein Fano toric variety $V_R$. A general anticanonical hypersurface in $V_R$ is a Calabi-Yau variety. We may resolve singularities of the ambient toric variety by refining the fan. In dimension $k \leq 4$, if we take a maximal simplicial refinement of $R$ (using all of the lattice points of $\diamond$), the resulting anticanonical hypersurface is a smooth $k-1$-dimensional Calabi-Yau manifold. In dimension $k=3$, the situation is even better: not only is the anticanonical hypersurface smoothed by maximal simplicial refinement, so is the ambient toric variety (see \cite[Corollary A.2.3]{CK}). Thus, three-dimensional reflexive polytopes yield smooth K3 surfaces in smooth toric varieties. A mirror pair of reflexive polytopes $\Delta$ and $\Delta^\circ$ yields mirror families of Calabi-Yau varieties.

One may describe a toric variety associated to a reflexive polytope by using \emph{generalized homogeneous coordinates}, so named because they generalize the homogeneous coordinates used in projective space. In this construction, if a fan $\Sigma$ has $q$ one-dimensional cones generated by polytope lattice points $\{v_1, \dots, v_q\}\}$, the toric variety $V_\Sigma$ is realized as a $k$-dimensional quotient of a subset of $\mathbb{C}^q$, and we have $q$ generalized homogeneous coordinates $z_1, \dots, z_q$. 

In more detail, for any subset $\mathcal{S}$ of $\{\rho_1, \dots, \rho_q\}\}$ such that $\mathcal{S}$ does not span a cone of $\Sigma$, we let $V(\mathcal{S}) \subseteq \mathbb{C}^q$ be the linear subspace defined by setting $z_j = 0$ for each cone $\rho_j \in \mathcal{S}$, and we let $Z(\Sigma)$ be the union of all such $V(\mathcal{S})$. Our toric variety will be a quotient of $\mathbb{C}^q - Z(\Sigma)$ by an appropriately chosen subset of $(\mathbb{C}^*)^q$, which acts on this subset by coordinatewise multiplication. Let us write each of our generating lattice points $v_j$ in coordinates as $(v_{j1}, \dots, v_{jk})$. Define a homomorphism $\phi_\Sigma : (\mathbb{C}^*)^q \to (\mathbb{C}^*)^k$ by
\begin{equation} 
	\phi_\Sigma(t_1, \dots, t_q) \mapsto \left( \prod_{j=1}^q t_j ^{v_{j1}} , \dots, \prod_{j=1}^q t_j^{v_{j k}} \right).
	\end{equation}

The toric variety $V_\Sigma$ associated with the fan $\Sigma$ is given by the quotient
\begin{equation} 
	V_\Sigma  = (\mathbb{C}^q - Z(\Sigma)) / \mathrm{Ker}(\phi_\Sigma).
		\end{equation}

\begin{example}
	Let $\Delta$ be the simplex with vertices $(1,  0,  0)$, $(0,  1,  0)$, $(0,  0,  1)$, and $(-3, -1, -1)$, and let $\Delta^\circ$ be its polar dual. Take the fan $R$ over the faces of $\Delta$. The one-dimensional cones of this fan are generated by $v_1 = (1,  0,  0), \dots, v_4 = (-3, -1, -1)$. The only subset of the vertices that does not span a face of $\Delta$, and therefore does not span a face of $R$, is $\{v_1, \dots, v_4\}$. Thus, $Z(R)$ is $\{(0,0,0,0)\}$. The homomorphism $\phi_R : (\mathbb{C}^*)^4 \to (\mathbb{C}^*)^3$ is given by $\phi_R(t_1, \dots, t_4) \mapsto \left(t_1 t_4^{-3}, t_2 t_4^{-1}, t_3 t_4^{-1} \right)$. The kernel of $\phi_R$ consists of elements of the form $(\lambda^3, \lambda, \lambda, \lambda)$ for $\lambda \in \mathbb{C}^*$, so $V_R$ is the weighted projective space $\mathbb{P}(3,1,1,1)$.
	
	Now, consider the fan $T$ over the faces of $\Delta^\circ$. In this case, $Z(T)$ is $\{(0,0,0,0)\}$ and the homomorphism $\phi_T : (\mathbb{C}^*)^4 \to (\mathbb{C}^*)^3$ is given by:
	\[\phi_R(t_1, \dots, t_4) \mapsto \left(t_1 t_2^{-1} t_3^{-1} t_4^{-1}, t_1^{-1} t_2^5 t_3^{-1} t_4^{-1}, t_1^{-1} t_2^{-1} t_3^5 t_4^{1} \right).\] 
	The kernel of $\phi_T$ is generated by elements of the form $(\lambda^3, \lambda, \lambda, \lambda)$ for $\lambda \in \mathbb{C}^*$ together with elements of the form $(1,\zeta,1,1)$ for $\zeta$ a cube root of unity. Thus, $V_T$ is the orbifold $\mathbb{P}(3,1,1,1)/C_3$, where $C_3$ is the multiplicative abelian group with 3 elements.
	
	By refining the fans $R$ and $T$ using other lattice points of $\Delta$ or $\Delta^\circ$, one can resolve singularities in the corresponding varieties.
\end{example}

Let us assume that $\Sigma$ is a maximal simplicial refinement of the fan $R$ over the faces of $\Delta$. In this case, the coordinates $z_i$ correspond to the non-origin lattice points of $\Delta$. We may use these coordinates to write explicit expressions for our Calabi-Yau hypersurfaces. Each lattice point $m$ in $\Delta^\circ$ will determine a monomial in the $z_i$. Let us choose a coefficient $\alpha_m$ associated to each monomial and use this information to define a polynomial associated to the vector of coefficients $\bm{\alpha}$:

\begin{equation}\label{E:fHat} \hat{f}_{\bm{\alpha}} = \sum_{m \; \in \; M \cap \Delta^\circ} \alpha_m \prod_{j=1}^q z_j^{\langle v_j, m \rangle + 1}.
\end{equation}

\noindent For a general choice of the $\alpha_m$, $\hat{f}_\alpha$ defines a smooth Calabi-Yau hypersurface. For computational convenience, we will often work with the Laurent polynomial $f_{\bm{\alpha}}$ obtained by restriction to the open torus. Explicitly, the Laurent polynomial is given by:

\begin{equation}\label{E:fLaurent} f_{\bm{\alpha}} = \sum_{m \; \in \; M \cap \Delta^\circ} \alpha_m \prod_{i=1}^k x_i^{\langle e_i, m \rangle},\end{equation}
where the $e_i$ are standard basis vectors.

We may obtain interesting subfamilies of hypersurfaces by specializing the $\alpha_m$ in combinatorially natural ways, a strategy pursued for example in \cite{KLMSW}. In particular, we may construct a one-parameter family by taking the sum of the monomials corresponding to the vertices of $\Delta^\circ$ and deforming by the monomial corresponding to the origin:

\begin{definition}\label{D:vertexPencil}
	Let $\Delta$ and $\Delta^\circ$ be a mirror pair of reflexive polytopes. The \emph{vertex pencil} $X_{\Delta, \psi}$ of Calabi-Yau hypersurfaces is the one-parameter family given by solutions to the equation:
	
	\[ \hat{f}_{\Delta,\psi} = \left(\sum_{x \; \in \; \mathrm{vertices}(\Delta^\circ)} \prod_{i=1}^k z_i^{\langle v_i, x \rangle + 1}\right) + \psi \prod_{i=1}^k z_i.\]
\end{definition}

If $\Delta$ is the simplex with vertices $(1,0,\dots, 0)$, \dots, $(0, \dots, 0, 1)$, $(-1,\dots,-1)$ in $\mathbb{R}^k$, then $\hat{f}_{\Delta,\psi}$ determines a one-parameter deformation of the Fermat hypersurface $z_1^{k+1}+\dots+z_{k+1}^{k+1}$; we view the vertex pencil as a combinatorial generalization of this construction.

One may use these explicit expressions for Calabi-Yau hypersurfaces to compute their Picard--Fuchs equations. 
Recall that a \emph{period} is the integral of a differential form with respect to a
specified homology class. In particular, periods of holomorphic forms encode the complex structure of
varieties. The Picard--Fuchs differential equation of a family of varieties
is a differential equation that describes the way the value of a
period changes as we move through the family.
Solutions to Picard--Fuchs equations for holomorphic forms on
Calabi-Yau varieties define the \emph{mirror map}, which relates variations of complex structure of a family to variations of complexified symplectic structure of the mirror family. The Fermat pencil and its Greene--Plesser mirror have the same Picard--Fuchs equation. One may generalize this observation using the notion of a mirror kernel pair.

\begin{lemma}\label{L:samePF}
	Let $\Delta$ and $\Gamma$ be a kernel pair of reflexive polytopes, and suppose $\Delta^\circ$ and $\Gamma^\circ$ are also a kernel pair. Then the Picard--Fuchs equations of the vertex pencils $X_{\Delta, \psi}$ and $X_{\Gamma, \psi}$ are the same.
\end{lemma}

\begin{proof}
	By \cite[Lemma 4.1]{MW}, if $R$ and $T$ are the fans over the faces of $\Delta$ and $\Gamma$ respectively, then there exists a toric variety $W$ and finite abelian subgroups of the torus $A$ and $B$ such that $V(R) = W/A$ and $V(T) = W/B$. Because $\Delta^\circ$ and $\Gamma^\circ$ are also a kernel pair, the polynomials $\hat{f}_{\Delta,\psi}$ and $\hat{f}_{\Gamma,\psi}$ determining the vertex pencils are identical. The vertex pencils may be different, due to the presence of finite quotient maps. But the Picard--Fuchs equation is preserved under finite quotient maps, so it is the same in all cases.
\end{proof}

Classical mirror symmetry posits a correspondence between complex and K\"{a}hler moduli spaces of Calabi-Yau varieties. A key part of this correspondence is the notion of a \emph{large complex structure limit point}. In the case of Calabi-Yau hypersurfaces in a Gorenstein Fano toric variety corresponding to a reflexive polytope $\Delta$, this point is determined by the origin of the polytope $\Delta^\circ$ in an appropriate moduli space construction (see \cite[Chapter 6]{CK}, \cite{HLZ}). Though the moduli space construction may in general be quite intricate, in the case of our vertex pencils, the complex structure limit point is simply $\psi=\infty$. We will make changes of variables as necessary in order to invoke results depending on the complex structure limit point.

\subsection{Kernel pairs in two and three dimensions}\label{S:kernelpairs}

In low dimensions, one may use the database of reflexive polytopes incorporated in \cite{sage} to classify kernel pairs. The two-dimensional case is discussed in \cite{MW} in the context of elliptic curve hypersurfaces. There are two mirror kernel pairs of reflexive triangles, a self-dual triangle, a kernel pair of quadrilaterals corresponding to $\mathbb{P}^1 \times \mathbb{P}^1$ and its mirror, a self-dual quadrilateral, a self-dual pentagon, and a self-dual hexagon.

In three dimensions, there are 417 polytopes in the list of 4319 reflexive polytopes that belong to a nontrivial kernel pair. Let us classify the mirror kernel pairs. The mirror kernel pairs may be clustered into larger sets of combinatorially equivalent polytopes with the same associated matrix kernels. We list the simplices, which are associated to weighted projective spaces and finite quotients of projective spaces, in Table~\ref{Ta:simplices}, using the numbering from the SageMath database of reflexive polytopes (see \cite{sage}).


\begin{table}[htb]
 \begin{tabular}{|c|c|} %
	\hline
	\textbf{Weights} & \textbf{Polytope pairs}\\
	\hline
$(1,1,1,1)$ & \makecell{$(0, 4311)$, $(8, 3313),$  $(427, 427)$, $(429, 429)$} \\
	\hline
	$(1,1,1,3)$ & \makecell{$(2, 4317)$, $(85, 3726),$  $(741, 1943)$} \\
	\hline
$(1,1,2,2)$ & \makecell{$(1, 4281)$,  $(742, 742)$, $(743, 744)$ }\\
	\hline
	$(1,1,2,4)$ &  \makecell{$(9, 4312), (428, 3315), (430, 3312), (431, 3314)$}\\
	\hline
	$(1,1,4,6)$ & \makecell{$(88, 4318), (1946, 3725)$} \\
	\hline
	$(1,2,2,5)$ & \makecell{$(31, 4255)$}\\
		\hline
	$(1, 2, 3, 6)$ & \makecell{$(89, 4228), (1944, 1948), (1947, 1947)$}\\
	\hline
	$(1, 2, 6, 9)$ & \makecell{$(745, 4282)$} \\
	 \hline
	$(1,3,4, 4)$ &  \makecell{$(87, 3727)$}\\
	\hline
	$(1,3, 8, 12)$ &  \makecell{$(1949, 4229)$} \\
	\hline
	$(1, 4, 5, 10)$ &  \makecell{$(1114, 3993)$} \\
	\hline
		$(1,6,14,21)$ &  \makecell{$(4080, 4080)$} \\
\hline
	 $(2, 3, 3, 4)$ & \makecell{$(86, 1945)$}\\
	 	\hline
	 $(2, 3, 10, 15)$  & \makecell{$(3038, 3038)$} \\
	 \hline
\end{tabular}
		\caption{Three-dimensional reflexive simplices and associated weights}
\label{Ta:simplices}
\end{table}

The remaining mirror kernel pairs fall into two groups. We follow \cite{MW} by referring to them as Group~I and Group~II. A representative for Group~I is given by reflexive polytope 3 in the \cite{sage} database, which has vertices $(1,0,0)$, $(0,1,0)$, $(0,0,1)$, $(-1,0,-1)$, and $(-1,-1,0)$; a representative for Group~II is given by reflexive polytope 10 in the \cite{sage} database, which has vertices $(1,0,0)$, $(0,1,0)$, $(0,0,1)$, $(-2,0,-1)$, and $(-2,-1,0)$. We list the polytopes in each group in Table~\ref{Ta:groups}.

\begin{table}[htb]
		\begin{tabular}{|c|c|} %
			\hline
			\textbf{Group} & \textbf{Polytope pairs}\\
			\hline
			Group I & \makecell{$(3, 4283)$,\\ $(753, 754)$} \\
			\hline
			Group II & \makecell{$(10, 4314)$,\\ $(433, 3316)$,\\ $(436, 3321)$}\\
			\hline			
		\end{tabular}
		\caption{Three-dimensional non-simplex mirror kernel pairs}
		\label{Ta:groups}
\end{table}



\section{Hasse-Witt matrices for Calabi-Yau toric hypersurfaces}\label{S:keylemma}

In this section, we combine results of \cite{HLYY} with the analysis of Picard--Fuchs equations in the previous section to prove the Key Lemma.

Katz gave an algorithm for the Hasse-Witt matrix of a Calabi--Yau hypersurface in projective space:

\begin{algorithm}[\cite{katz} (2.3.7.17)]
	Let $X$ be a smooth Calabi-Yau hypersurface of degree $d = n+1$ in $\PP^{n}$ that is given by an equation $\hat{f}$. Then the Hasse-Witt matrix $\mathrm{HW}_p(X)$ is given by the coefficient of $(z_0 \cdots z_{n})^{p-1}$ in $\hat{f}^{p-1}$. 
\end{algorithm}

Huang, Lian, Yau, and Yu extended this algorithm to the case of a Calabi--Yau hypersurface in a smooth toric variety, writing their result in terms of Laurent polynomials.

\begin{algorithm}[\cite{HLYY} (Corollary 2.3)]\label{A:HLYY}
	Let $f_{\bm{\alpha}}$ be a Laurent polynomial determining a smooth toric Calabi--Yau hypersurface $X$. Then the Hasse-Witt matrix $\mathrm{HW}_p(X)$ is given by the coefficient of the constant term of $f_{\bm{\alpha}}^{p-1}$.
\end{algorithm}

\noindent The constant term of $f_{\bm{\alpha}}^{p-1}$ corresponds to the $(z_0 \cdots z_{q})^{p-1}$ term in generalized homogeneous coordinates. Samol and van Straten studied congruences for the constant term of $f_{\bm{\alpha}}^{p-1}$ in \cite{samolstraten}, using an analogy to periods that the results of \cite{HLYY} and \cite{BV} make precise. We emphasize that although Algorithm~\ref{A:HLYY} is phrased in terms of the Laurent polynomial for computational convenience, the underlying geometric object is a Calabi--Yau hypersurface in a complete toric variety, not merely the big torus.

Huang, Lian, Yau, and Yu showed that the Hasse-Witt matrix of a Calabi-Yau variety realized as a hypersurface in a smooth toric variety can be described $\pmod{p}$ as the truncation of a series expansion of the period integral.

\begin{theorem}[\cite{HLYY} (Theorem 1.2)]\label{T:HLYY}
	Let $f_{\bm{\alpha}}$ be a family of Laurent polynomials determining a family of toric Calabi-Yau hypersurfaces, as in Equation~\ref{E:fLaurent}. Suppose $\bm{\gamma}$ is a large complex structure limit point, that is, $f_{\bm{\gamma}}$ has maximally unipotent monodromy.
	\begin{enumerate}
		\item The Hasse-Witt matrix $\mathrm{HW}_p$ is a polynomial in $\alpha_m$ of degree $p-1$.
		\item The period integral $\mathcal{I}$ for the holomorphic form can be extended as a holomorphic function at ${\bm{\gamma}}$ and has the form $\frac{1}{\alpha_{\bm{0}}} T\left(\frac{\alpha_m}{\alpha_{\bm{0}}}\right)$, where $T\left(\frac{\alpha_m}{\alpha_{\bm{0}}}\right)$ is a Taylor series in ${\{ \frac{\alpha_m}{\alpha_{\bm{0}}} \mid m \neq \bm{0} \}}$ with integer coefficients.
		\item The Hasse-Witt matrix satisfies the truncation relation
		\[\mathrm{HW}_p \equiv \left[ T\left(\frac{\alpha_m}{\alpha_{\bm{0}}}\right) \right]^{p-1} \pmod{p},\]
		where $[\cdot]^{p-1}$ represents the truncation of a series at degree $p-1$. 
	\end{enumerate}
\end{theorem}

We now apply Theorem~\ref{T:HLYY} to vertex pencils obtained from combinatorially equivalent reflexive polytopes to prove our Key Lemma.
\begin{lemma}[Key Lemma]\label{T:mainProof}
Let $\Delta$ and $\Gamma$ be a kernel pair of $n$-dimensional reflexive polytopes, and suppose $\Delta^\circ$ and $\Gamma^\circ$ are also a kernel pair. Let $V_\Delta$ and $V_\Gamma$ be smooth toric varieties determined by maximal simplicial refinements of the fans over the faces of $\Delta$ and $\Gamma$, respectively. Let $X_{\Delta,\psi}$ and $X_{\Gamma, \psi}$ be the corresponding vertex pencils. Then for any rational $\psi$ and prime $p$ such that $X_{\Delta,\psi}$ and $X_{\Gamma, \psi}$ are smooth, their Hasse--Witt matrices are the same.
\end{lemma}

\begin{proof}
Let $B_\Delta^\circ$ and $B_\Gamma^\circ$ be the matrices whose columns are given by the vertices of the polytopes $\Delta^\circ$ and $\Gamma^\circ$, respectively. Consider the diagonal pencils $f_\Delta$ and $f_\Gamma$. We may write $f_{\Delta, \psi}$ as the Laurent polynomial $\bm{z}^{w_1} + \dots + \bm{z}^{w_\ell}+\psi \bm{z}^{\bm{0}}$ and $\bm{z}^{w_1} + \dots + \bm{z}^{w_\ell}+\psi \bm{z}^{\bm{0}}$ and $f_{\Gamma, \psi}$ as the Laurent polynomial $\bm{z}^{w_1'} + \dots + \bm{z}^{w_\ell'}+\psi \bm{z}^{\bm{0}}$, where $\{w_1, \dots, w_\ell\}$ are the vertices of $\Delta^\circ$ and $\{w_1', \dots, w_\ell'\}$ are the vertices of $\Gamma^\circ$. By Algorithm~\ref{A:HLYY}, $\mathrm{HW}_p(X_{\Delta,\psi})$ is given by the constant term of $(\bm{z}^{w_1} + \dots + \bm{z}^{w_\ell}+\psi \bm{z}^{\bm{0}})^{p-1}$. Contributions to the constant term are given by nonnegative integers $\alpha_1, \dots, \alpha_\ell$ such that $\alpha_1 w_1 + \dots + \alpha_\ell w_\ell = 0$. Thus, the vector $(\alpha_1, \dots, \alpha_\ell)$ is an element of the kernel of $B_\Delta^\circ$. Meanwhile, $\mathrm{HW}_p(X_{\Gamma,\psi})$ is given by the constant term of $(\bm{z}^{w_1'} + \dots + \bm{z}^{w_\ell'}+\psi \bm{z}^{\bm{0}})^{p-1}$. We see that the vector $(\alpha_1, \dots, \alpha_\ell)$ yields a contribution to $\mathrm{HW}_p(X_{\Delta,\psi})$ if and only if it yields a contribution to $\mathrm{HW}_p(X_{\Gamma,\psi})$.
\end{proof}

Lemma~\ref{T:mainProof} immediately yields a relationship between point counts $\pmod{p}$.

\begin{corollary}\label{C:pointCountsModp}
Let $\Delta$ and $\Gamma$ be a kernel pair of $n$-dimensional reflexive polytopes, and suppose $\Delta^\circ$ and $\Gamma^\circ$ are also a kernel pair. Let $V_\Delta$ and $V_\Gamma$ be smooth toric varieties determined by maximal simplicial refinements of the fans over the faces of $\Delta$ and $\Gamma$, respectively. Let $X_{\Delta,\psi}$ and $X_{\Gamma, \psi}$ be the corresponding vertex pencils. Then for any rational $\psi$ and prime $p$ such that $X_{\Delta,\psi}$ and $X_{\Gamma, \psi}$ are smooth, 
\[\# X_{\Delta,\psi} \equiv \# X_{\Gamma, \psi} \pmod{p}.\]
\end{corollary}

One common situation where the hypotheses of Theorem~\ref{T:mainProof} and Corollary~\ref{C:pointCountsModp} apply is that of reflexive simplices admitting a toric resolution of singularities. Up to resolution of singularities, the toric variety determined by a reflexive simplex is either a Gorenstein Fano weighted projective space or a quotient of such a projective space by a finite abelian group. In particular, Theorem~\ref{T:mainProof} applies to the Greene--Plesser mirrors of K3 hypersurfaces in Gorenstein Fano weighted projective spaces. We suspect that the techniques of \cite{HLYY} may be modified to apply to resolutions of singularities of Calabi-Yau hypersurfaces in Gorenstein Fano weighted projective spaces that are not themselves smooth, though pursuing such a hypothesis is beyond the scope of the current work.

\begin{example}\label{Ex:weighted}
	The weighted projective space $\mathbb{WP}(1,2,2,5)$ is a Gorenstein Fano threefold. As a toric variety, this space is determined by the three-dimensional reflexive polytope $\Delta$ with vertices $(1,0,0)$, $(0,1,0)$, $(0,0,1)$, and $(-5,-2,-2)$. The polar dual $\Delta^\circ$ has vertices $(1, -1, -1)$, $(-1, 4, -1)$, $(-1, -1,  4)$, and $(-1, -1, -1)$. $\Delta$ and $\Delta^\circ$ are combinatorially equivalent, and the kernels of the matrices whose columns are given by their vertices form the same submodule of $\mathbb{Z}^4$.
	
	The pencil of K3 surfaces $X_{\Delta,\psi}$ is a resolution of singularities of the diagonal pencil given by solutions to $z_0^{10}+z_1^{5}+z_2^{5}+z_3^2+\psi z_0\cdots z_3=0$. Let $G$ be the group of diagonal symmetries of the pencil that preserve the holomorphic $(2,0)$-form. Then $G$ is a finite abelian group isomorphic to $\mathbb{Z}/5$. The Greene--Plesser mirror of $X_{\Delta,\psi}$ is the resolution of singularities of the quotient $X_{\Delta,\psi}/G$. But we may also obtain this mirror by taking the diagonal pencil $X_{{\Delta^\circ},\psi}$ within the toric variety corresponding to $\Delta^\circ$. We conclude that, for any rational $\psi$ and prime $p$ such that both $X_{\Delta,\psi}$ and its Greene-Plesser mirror are smooth, both threefolds have the same Hasse--Witt matrix and $\# X_{\Delta,\psi} \equiv X_{{\Delta^\circ}, \psi} \pmod{p}$.
\end{example}

Lemma~\ref{T:mainProof} is powered by the relationship between periods and point counts over finite fields. The periods of a general anticanonical Calabi-Yau hypersurface in a Gorenstein Fano toric variety satisfy a GKZ generalized hypergeometric system. When we specialize to a vertex pencil, we obtain varieties with special geometric and arithmetic properties. In the next section, we study specific examples of elliptic curve and K3 surface vertex pencils whose Picard--Fuchs equations are classical hypergeometric differential equations, and remark on the relationships between geometric and arithmetic properties in this context.

\section{Proof of the Main Theorem}\label{S:examples}

We are now ready to prove our Main Theorem. First, we need polytopes satisfying the hypothesis of our Key Lemma \ref{T:main}. In Section \ref{S:kernelpairs}, we classified all kernel pairs in three dimensions. The polytopes listed as Group I and Group II in Table \ref{Ta:groups} determine vertex pencils of K3 surfaces whose general member has Picard rank 19 over $\mathbb{C}$; this corresponds to a natural one-parameter complex deformation space. Of the polytopes listed in Table \ref{Ta:simplices}, only those in weights $(1,1,1,1)$ and $(1,1,1,3)$ determine vertex pencils with this property. (The vertex pencils corresponding to the other polytopes in the table fit naturally into multiparameter families.)  


In two dimensions, the most interesting case is the cross-polytope with vertices $(\pm 1,0)$ and $(0, \pm 1)$, and its polar dual. Geometrically, the two-dimensional cross-polytope determines a family of elliptic curves in $\mathbb{P}^1 \times \mathbb{P}^1$, related to the Legendre family. We begin by studying this vertex pencil. In doing so we establish the techniques for proving our Main Theorem and remark on the connection to classical work of Igusa \cite{igusa}.  

We next prove the Main Theorem, i.e., we compute the Hasse-Witt matrix as a truncated hypergeometric function for the four types of Picard rank 19.  We follow the approach of \cite[Algorithm 3]{salerno}, which allows us compute the hypergeometric parameters of a given Picard--Fuchs equation. 

Picard-rank 19 K3 surfaces are very special. We make some remarks on their connection to elliptic curves, related hypergeometric functions, and modularity. 

\subsection{Elliptic curves in $\mathbb{P}^1 \times \mathbb{P}^1$}\label{elliptic}

Let $\Delta$ be the cross-polytope with vertices $(\pm 1,0)$ and $(0, \pm 1)$. The vertex pencil $X_{\Delta, \psi}$ is a pencil of elliptic curves in $\mathbb{P}^1 \times \mathbb{P}^1$. Applying the Griffiths--Dwork method, one may show that its Picard--Fuchs equation is given by: 

\begin{equation}\label{E:ecPF}
\frac{\psi}{(\psi^3-16\psi)} F(\psi) +  \frac{(3\psi^2 - 16)}{(\psi^3 - 16\psi)}F'(\psi) + F''(\psi)=0.
\end{equation}

\begin{proposition} The number of points of $X_{\Delta, \psi}$ is
\[N_{\F_p}(\mathbb{P}_1\times\mathbb{P}_1)\equiv 1+ \left[{}_2F_{1}\left(\frac{1}{2},\frac{1}{2};1\mid \frac{\psi^2}{16}\right)\right]^{(p-1)}\bmod p,\]
where $[ \cdot ]^{n}$ denotes the truncation of the series at the $n$-th term. 

\end{proposition}

\begin{proof}

The companion matrix to the Picard--Fuchs equation~\ref{E:ecPF} is 

$$\left(\begin{array}{cc} 0&1\\ \dfrac{-\psi}{(\psi^3-16\psi)} & \dfrac{(-3\psi^2 + 16)}{(\psi^3 - 16\psi)}\end{array}\right).$$

This represents a linear system of differential equations of the form $\dfrac{d}{d\psi}y=A(\psi)y(\psi)$.

This system does not necessarily have regular singular points at $0$ and $\infty$, but a simple change of basis gives an equivalent system which does (c.f. \cite{beukers}). 

We can thus write the system with a new matrix, given by

\begin{eqnarray*}\left(\begin{array}{cc} 0&1\\ \dfrac{-\psi^2}{(\psi^3-16\psi)} & 1+ \dfrac{(-3\psi^2 + 16)}{(\psi^3 - 16\psi)}\end{array}\right)&=& \left(\begin{array}{cc} 0&1\\ \dfrac{-\psi^2}{(\psi^3-16\psi)} & \dfrac{-2\psi^2}{(\psi^3 - 16\psi)}\end{array}\right)\\
&=& \frac{1}{\psi} \left(\begin{array}{cc} 0&1\\ \dfrac{-\psi^2}{(\psi^2-16)} & \dfrac{-2\psi^2}{(\psi^2 - 16)}\end{array}\right)\end{eqnarray*}

Changing variables so that $z=\psi^2$, we get 

$$\frac{1}{z} \left(\begin{array}{cc} 0&1/2\\ \dfrac{-z}{2(z-16)} & \dfrac{-z}{(z - 16)}\end{array}\right)$$

Now we can change variables so that $z=16\lambda$, and we get

$$\frac{1}{\lambda} \left(\begin{array}{cc} 0&1/2\\ \dfrac{-\lambda}{2(\lambda-1)} & \dfrac{-\lambda}{(\lambda - 1)}\end{array}\right)$$

The system defined by this matrix has regular singular points at $0,1,\infty$. The eigenvalues of the residue at $0$ will give us the $\beta$ parameters of the hypergeometric function and the eigenvalues of the residue at $\infty$ will give us the $\alpha$ parameters.

The residue at $\lambda=0$ is 

$$\left(\begin{array}{cc} 0&1/2\\ 0 & 0 \end{array}\right)$$

The eigenvalues are clearly $0,0$, and thus $\beta=\{0,0\}$. 

We change variables to get a system with simple pole at $\infty$, and get

$$\frac{1}{\zeta} \left(\begin{array}{cc} 0&-1/2\\ \dfrac{1}{2(1-\zeta)} & \dfrac{1}{(1-\zeta)}\end{array}\right)$$

The residue at $\zeta=0$ is 

$$\left(\begin{array}{cc} 0&-1/2\\ 1/2 & 1 \end{array}\right)$$

which has eigenvalues $\alpha=\{1/2,1/2\}$. 

Thus, this system is related to a hypergeometric function of the form

$${}_2F_{1}\left(\frac{1}{2},\frac{1}{2};1\mid \frac{\psi^2}{16}\right).$$

\begin{remark} \label{R:mum} An $n$-th order Picard Fuchs equation has maximally unipotent monodromy if and only if it has indicial equation of the form $(\lambda -\ell)^n$ (see \cite[\S 5.1]{CK} for a discussion). In the hypergeometric case, this condition is equivalent to requiring that the $\beta$ parameters be integers. 
\end{remark}

Notice that this monodromy group (with denominator parameter 1) does have this desired property. 

Finally, by Theorem \ref{T:mainProof} we now know that the truncated hypergeometric function is our Hasse--Witt invariant, and that the number of points is

\[N_{\F_p}(\mathbb{P}_1\times\mathbb{P}_1)\equiv 1+ \left[{}_2F_{1}\left(\frac{1}{2},\frac{1}{2};1\mid \frac{\psi^2}{16}\right)\right]^{(p-1)}\bmod p.\]

\end{proof}

\begin{remark}
This hypergeometric function appears in many classical results. In Igusa's work \cite{igusa}, it is geometrically related to the Legendre family of elliptic curves, and according to  \cite{Harnad}, based on work by \cite{Conway}, it is also associated to the modular group $\Gamma_0(4)$. 
\end{remark}

\subsection{Fermat quartic in $\mathbb{P}^3$}\label{sec:fermat}

The Fermat quartic is the K3 surface defined by 

\begin{equation}\label{eqn:fermat}
X \colon x_0^4+x_1^4+x_2^4+x_3^4-4\psi x_0x_1x_2x_3=0.
\end{equation}

Its Picard-Fuchs equation and Hasse-Witt matrix have been studied extensively, starting with the seminal work of Dwork \cite[\S6]{padic}. 
In \cite{salerno:counting}, using $p$-adic methods, we obtain that 

\[N_{\F_p}(X)\equiv 1+ \left[{}_3F_2\left(\frac{1}{2},\frac{1}{4},\frac{3}{4};1,1\mid \frac{1}{\psi^4}\right)\right]^{(p-1)}\bmod p.\]

\subsection{Sextic in $\mathbb{P}(1,1,1,3)$}\label{sec:weighted}

We consider the K3 surface family defined by

\begin{equation}\label{eqn:weighted}
 x_0^2+x_1^6+x_2^6+x_3^6-\psi x_0x_1x_2x_3=0.
\end{equation}

The Picard-Fuchs differential equation for this pencil is 

\begin{equation}\label{PFweighted}
F'''(\psi)+\frac{6\psi^6+5184}{\psi^7-1728\psi}F''(\psi)+\frac{7\psi^6-5184}{\psi^8-1728\psi^2}F'(\psi)+\frac{\psi^3}{\psi^6-1728}F(\psi)=0.\end{equation}

\begin{proposition}\label{P:weighted}
Let $X_{\Delta, \psi}$ be the vertex pencil of K3 surfaces in \ref{eqn:weighted}.  Then,
\[N_{\F_p}(X)\equiv 1+ \left[{}_3F_2\left(\frac{1}{2},\frac{1}{6},\frac{5}{6};1,1\mid \frac{1728}{\psi^6}\right)\right]^{(p-1)}\bmod p.\]
\end{proposition}

\begin{proof}

As in Section \ref{elliptic}, we first write the companion matrix of the Picard--Fuchs differential equation:

\[\left(\begin{array}{ccc}0&1&0\\
0&0&1\\
\dfrac{-\psi^3}{\psi^6-1728}&\dfrac{-7\psi^6+5184}{\psi^2(\psi^6-1728)}&\dfrac{-6\psi^6-5184}{\psi(\psi^6-1728)}\end{array}\right)\]

A change of basis yields a linear system of differential equations with regular singular points at $0$ and $\infty$:

\[\frac{1}{\psi}\left(\begin{array}{ccc}0&1&0\\
0&1&1\\
\dfrac{-\psi^6}{\psi^6-1728}&\dfrac{-7\psi^6+5184}{\psi^6-1728}&2+\dfrac{-6\psi^6-5184}{\psi^6-1728}\end{array}\right)\]\[=\frac{1}{\psi}\left(\begin{array}{ccc}0&1&0\\
0&1&1\\
\dfrac{-\psi^6}{\psi^6-1728}&\dfrac{-7\psi^6+5184}{\psi^6-1728}&\dfrac{-4\psi^6-8640}{\psi^6-1728}\end{array}\right)\]

We do a change of variables so that $z=\psi^6$:

\[\frac{1}{z}\left(\begin{array}{ccc}0&1/6&0\\
0&1/6&1/6\\
\dfrac{-z}{6(z-1728)}&\dfrac{-7z+5184}{6(z-1728)}&\dfrac{-4z-8640}{6(z-1728)}\end{array}\right)\]

And finally, we do a change of variables so that the system has regular singular points around 1, by letting $\lambda=\frac{z}{1728}$:

\[\frac{1}{\lambda}\left(\begin{array}{ccc}0&1/6&0\\
0&1/6&1/6\\
\dfrac{-\lambda}{6(\lambda-1)}&\dfrac{-7\lambda+3}{6(\lambda-1)}&\dfrac{-4\lambda-5}{6(\lambda-1)}\end{array}\right)\]

And now we can compute the hypergeometric denominator parameters by computing the eigenvalues of the residue at $0$:

\[\left(\begin{array}{ccc}0&1/6&0\\
0&1/6&1/6\\
0&-1/2&5/6\end{array}\right),\]

\noindent which are $\beta=\{0,1/3,2/3\}$. A simple change of variables gives us the system centered at $\infty$, by letting $\xi=\frac{1}{z}$:

\[\frac{1}{\xi}\left(\begin{array}{ccc}0&-1/6&0\\
0&-1/6&-1/6\\
\dfrac{1}{6(1-\xi)}&\dfrac{-3\xi+7}{6(1-\xi)}&\dfrac{5\xi+4}{6(1-\xi)}\end{array}\right)\]

And so we obtain the residue at $\infty$

\[\left(\begin{array}{ccc}0&-1/6&0\\
0&-1/6&-1/6\\
1/6&7/6&2/3\end{array}\right),\]

\noindent  whose eigenvalues give us the numerator parameters, $\alpha=\{1/6,1/6,1/6\}$.

Thus, the associated hypergeometric function is \[{}_3F_2\left(\frac{1}{6},\frac{1}{6},\frac{1}{6};\frac{1}{3},\frac{2}{3}\mid \frac{\psi^6}{1728}\right).\]

By Theorem \ref{T:HLYY}, we require that the Picard--Fuchs equation  have maximally unipotent monodromy around $0$. By Remark \ref{R:mum}, this is clearly not the case in the above computation. However, the change of variables sending $z$ to $1/\zeta$ provides us with another hypergeometric differential equation that does have the desired monodromy. In this particular instance, that hypergeometric function is \[{}_3F_2\left(\frac{1}{2},\frac{1}{6},\frac{5}{6};1,1\mid \frac{1728}{\psi^6}\right).\]

And thus, the point count is 

\[N_{\F_p}(X)\equiv 1+ \left[{}_3F_2\left(\frac{1}{2},\frac{1}{6},\frac{5}{6};1,1\mid \frac{1728}{\psi^6}\right)\right]^{(p-1)}\bmod p.\]
\end{proof}

\subsection{Group I}\label{sec:groupI}

We now consider K3 surface examples associated to the set of reflexive polytopes identified in \cite{MW} as Group~I. This group contains four combinatorially equivalent three-dimensional polytopes: reflexive polytope 3 in the \cite{sage} database, with vertices $(1,0,0)$, $(0,1,0)$, $(0,0,1)$, $(-1,0,-1)$, and $(-1,-1,0)$, its polar dual polytope 4283, and the pair of polar polytopes 753 and 754. The kernel of the matrix whose columns are given by the vertices of $\Delta$ is the same submodule of $\mathbb{Z}^5$ for any polytope $\Delta$ in Group~I; since the set is closed under polar duality, the same condition holds for polar polytopes $\Delta^\circ$.

For any $\Delta$ in Group~I, the Picard-Fuchs differential equation for the vertex pencil of K3 surfaces $X_{\Delta, \psi}$ is

\begin{equation}\label{PFI}
F'''(\psi)+\frac{6(\psi^3+27)}{\psi^4+108\psi}F''(\psi)+\frac{7\psi^2}{\psi^4+108\psi}F'(\psi)+\frac{\psi}{\psi^4+108\psi}F(\psi)=0.\end{equation}

\begin{proposition}\label{P:groupI}
Let $X_{\Delta, \psi}$ be the vertex pencil of K3 surfaces associated to the three-dimensional reflexive polytope with index 3, 4283, 753, or 784. Then,
\[N_{\F_p}(X)\equiv 1+ \left[{}_3F_2\left(\frac{1}{2},\frac{1}{3},\frac{2}{3};1,1\mid -\frac{108}{\psi^3}\right)\right]^{(p-1)}\bmod p.\]
\end{proposition}

\begin{proof}

Using the same approach described for Proposition \ref{P:weighted}, one finds that the Picard--Fuchs equation is equivalent to the hypergeometric differential equation satisfied by 
\[{}_3F_2\left(\frac{1}{3},\frac{1}{3},\frac{1}{3};\frac{1}{3},\frac{1}{6}\mid -\frac{\psi^3}{108}\right).\]

Again, the associated Picard--Fuchs equation  does not have maximally unipotent monodromy around $0$. By the same change of variables we can resolve this, and obtain a hypergeometric differential equation that does have the desired monodromy. In this particular instance, that hypergeometric function is \[{}_3F_2\left(\frac{1}{2},\frac{1}{3},\frac{2}{3};1,1\mid -\frac{108}{\psi^3}\right).\]

And thus, the point count is 

\[N_{\F_p}(X)\equiv 1+ \left[{}_3F_2\left(\frac{1}{2},\frac{1}{3},\frac{2}{3};1,1\mid -\frac{108}{\psi^3}\right)\right]^{(p-1)}\bmod p.\]

\end{proof}

\subsection{Group II }\label{sec:groupII}

Similarly, we may analyze K3 surface examples associated to the set of reflexive polytopes identified in \cite{MW} as Group~II. This group contains six combinatorially equivalent three-dimensional polytopes: reflexive polytope 10 in the \cite{sage} database, with vertices $(1,0,0)$, $(0,1,0)$, $(0,0,1)$, $(-2,0,-1)$, and $(-2,-1,0)$, its polar dual polytope 4314, the pair of polar polytopes 433 and 3316, and the pair of polar polytopes 436 and 3321. As before, the kernel of the matrix whose columns are given by the vertices of $\Delta$ is the same submodule of $\mathbb{Z}^5$ for any polytope $\Delta$ in Group~II, and the set is closed under polar duality.

For any $\Delta$ in Group~II, the Picard-Fuchs differential equation for the vertex pencil of K3 surfaces $X_{\Delta, \psi}$ is

\begin{equation}\label{PFII}
F'''(\psi)+\frac{6(\psi^3)}{\psi^4-256}F''(\psi)+\frac{7\psi^2}{\psi^4-256}F'(\psi)+\frac{\psi}{\psi^4-256}F(\psi)=0.\end{equation}

 \begin{proposition}
Let $X_{\Delta, \psi}$ be the vertex pencil of K3 surfaces associated to the three-dimensional reflexive polytope with index 10, 4314, 433, 3316, 436, or 3321. Then 
\[N_{\F_p}(X)\equiv 1+ \left[{}_3F_2\left(\frac{1}{2},\frac{1}{4},\frac{3}{4};1,1\mid \frac{256}{\psi^4}\right)\right]^{(p-1)}\bmod p.\]

\end{proposition}

\begin{proof}
Again, using the same approach described for Proposition \ref{P:weighted}, one finds that the Picard--Fuchs equation is equivalent to the hypergeometric differential equation satisfied by 
\[{}_3F_2\left(\frac{1}{4},\frac{1}{4},\frac{1}{4};\frac{1}{2},\frac{3}{4}\mid \frac{\psi^4}{256}\right).\]

The hypergeometric differential equation with maximally unipotent monodromy around $0$ corresponds to the hypergeometric function \[{}_3F_2\left(\frac{1}{2},\frac{1}{4},\frac{3}{4};1,1\mid \frac{256}{\psi^4}\right).\]

And thus, the point count is 

\[N_{\F_p}(X)\equiv 1+ \left[{}_3F_2\left(\frac{1}{2},\frac{1}{4},\frac{3}{4};1,1\mid \frac{256}{\psi^4}\right)\right]^{(p-1)}\bmod p.\]

\end{proof}

\subsection{Symmetric square roots}

Our vertex pencils are combinatorially natural one-parameter families of K3 surfaces. The vertex pencils associated to Group~I or Group~II all have Picard rank 19. This property is special: a general K3 hypersurface in the toric variety obtained from polytope 3 (in Group~I) or polytope 10 (in Group~II) would have Picard rank 2, for example (see \cite{k3expository} for a discussion of such computations). The Picard--Fuchs equations of Picard rank 19 K3 surfaces have particularly nice properties:

\begin{theorem}\cite{Doran}\label{T:DSymmSquare} The Picard--Fuchs equation of a family of rank-$19$ lattice-polarized K3 surfaces can be written as the symmetric square of a second-order homogeneous linear Fuchsian differential equation. 
\end{theorem}

In our case, one can obtain the conclusion of Theorem~\ref{T:DSymmSquare} a bit more quickly. A classical result on hypergeometric functions known as \emph{Clausen's formula} (cf. \cite{Bailey}) states that 

\[{}_2F_1\left(a,b;a+b+\frac{1}{2}|z\right)^2={}_3F_2\left(2a,2b,a+b;2a+2b,a+b+\frac{1}{2}|z\right).\]

Applying this to our two examples, we find that for the Group I hypergeometric functions,

\begin{equation}\label{GPI}{}_2F_1\left(\frac{1}{6},\frac{1}{3};1\mid-\frac{108}{\psi^3}\right)^2={}_3F_2\left(\frac{1}{2},\frac{1}{3},\frac{2}{3};1,1\mid -\frac{108}{\psi^3}\right),\end{equation}

\noindent and for Group II,

\begin{equation}\label{GPII}{}_2F_1\left(\frac{1}{8},\frac{3}{8};1|\frac{256}{\psi^4}\right)^2={}_3F_2\left(\frac{1}{2},\frac{1}{4},\frac{3}{4};1,1\mid \frac{256}{\psi^4}\right).\end{equation}

%
%
%
%
%
%
%
%

The equations (\ref{GPI}) and (\ref{GPII}) relate our K3 surfaces to some very special Gaussian hypergeometric functions. These hypergeometric functions are related to classical modular groups. In a process similar to the elliptic curves case, by consulting \cite{Harnad} and \cite{Conway}, we see that 

\[{}_2F_1\left(\frac{1}{6},\frac{1}{3};1|\psi\right)\]

\noindent is associated to the modular group $\Gamma_0(3)+$ (the $+$ meaning it has all of its Atkin Lehner involutions).

Finally, the hypergeometric function $${}_2F_1\left(\frac{1}{8},\frac{3}{8};1|\psi\right)$$ is associated with the modular group $\Gamma_0(2)+$.

\begin{remark} 
It would be tempting to say that Clausen's formula or Theorem \ref{T:DSymmSquare} imply that the point counts for these two examples are quadratic residues $\pmod p$, but the true situation is more subtle. Due to translations between truncated generalized hypergeometric functions and their finite field analogs, this only works for certain primes. For example, for Group II and the prime $13$, the point count is not a quadratic residue for $\psi=1, 5, 8, 12$. 
\end{remark}

\subsection{Special hypergeometric parameters}\label{S:fourcases}

Let us recall the following technical condition on hypergeometric function parameters:

\begin{definition}[\cite{BCM}]\label{D:overQ}
	We say the hypergeometric function given by the parameters $(\alpha_1, \dots, \alpha_d)$ and $(\beta_1, \dots, \beta_{d-1})$ is \emph{defined over $\mathbb{Q}$} if the coefficients of the polynomials $\prod_{j=1}^d (x-e^{2\pi i \alpha_j})$ and $\prod_{j=1}^{d-1} (x-e^{2\pi i \beta_j})$ are in $\mathbb{Q}$.
\end{definition}

Hypergeometric functions defined over $\mathbb{Q}$ satisfy many nice properties. The multisets of parameters corresponding to ${}_3F_2$ hypergeometric functions defined over $\mathbb{Q}$ with maximally unipotent monodromy were classified in \cite{RV}. The possible $\alpha_j$ parameters are $(\frac{1}{2},\frac{1}{4},\frac{3}{4})$, $(\frac{1}{2},\frac{1}{6},\frac{5}{6})$, $(\frac{1}{2},\frac{1}{3},\frac{2}{3})$, and $(\frac{1}{2},\frac{1}{2},\frac{1}{2})$; because we require maximally unipotent monodromy at the origin, the $\beta_j$ parameters are always $(1,1)$. 

We identified three of the four multisets of $\alpha_j$ parameters as the Picard-Fuchs equations of vertex pencils in our main theorem. The question of whether we can obtain the remaining multiset of parameters in a similar fashion naturally arises. In fact, the remaining parameters can also be obtained as the Picard-Fuchs equation of a vertex pencil of K3 surfaces studied by the second author and several collaborators:

\begin{lemma}[\cite{KLMSW}]\label{L:skewoctahedron}
	Let $\Delta$ be the reflexive octahedron with vertices $(1, 1, 1)$, $(-1,-1, 1)$, $(-1, 1,-1)$, $(1,-1, 1)$, $(1, 1,-1)$, and $(-1,-1,-1)$. The Picard-Fuchs equation for the vertex pencil of $\Delta$ satisfies
	\begin{equation}
		\frac{d^3 \omega}{d t^3} + \frac{6(t^2-32)}{t(t^2-64)}\frac{d^2 \omega}{d t^2} + \frac{7t^2-64}{t^2(t^2-64)}\frac{d \omega}{d t} + \frac{1}{t(t^2-64)}\omega,
	\end{equation}
where we have made the change of variables $t=1/\psi$.
\end{lemma}

Using the methods described in the previous sections, one can show that this Picard-Fuchs equation is hypergeometric with parameters ${}_3F_2\left(\frac{1}{2},\frac{1}{2},\frac{1}{2};1,1\mid \frac{t^2}{64}\right)$. Thus, we have proved the following geometric realization result:

\begin{corollary}\label{C:4multisets}
	Each of the four multisets of hypergeometric parameters yielding ${}_3F_2$ hypergeometric functions defined over $\mathbb{Q}$ with maximally unipotent monodromy at the origin corresponds to the Picard-Fuchs equation of a vertex pencil of K3 surfaces realized as hypersurfaces in a Gorenstein Fano toric variety.
\end{corollary}


\begin{thebibliography}{DKSSVW18a}

\bibitem[AH19]{AH}
Achter, J., Howe, E.: Hasse-{W}itt and {C}artier-{M}anin matrices: a warning and a request, in \emph{Arithmetic geometry: computation and applications}, Contemp. Math. 722, Amer. Math. Soc., Providence, RI, 1--18 (2019).


\bibitem[AS16]{AS} Adolphson, A., Sperber, S.: {$A$}-hypergeometric series and the {H}asse-{W}itt matrix of a
 hypersurface, Finite Fields Appl. \textbf{41}, 55--63 (2016).

\bibitem[AP15]{AP} Aldi, M., Peruni\v{c}i\'c, A.: $p$-adic Berglund-H\"ubsch duality. Adv. Theor. Math. Phys. 
\textbf{19}, no. 5, 1115-1139 (2015).

\bibitem[Bai64]{Bailey} Bailey, W.N.: Generalized hypergeometric series. Cambridge Tracts in Mathematics
and Mathematical Physics, No. 32. Stechert-Hafner, Inc., New York (1964).

\bibitem[Beu09]{beukers}
Beukers, F.: Notes on differential equations and
hypergeometric functions, Course Lecture Notes (2009).

\bibitem[BCM15]{BCM}
Beukers, Frits and Cohen, Henri, Mellit, Anton: Finite hypergeometric functions, Pure Appl. Math. Q. \textbf{11}, 559--589 (2015).

\bibitem[BH89]{BH} Beukers, F., Heckman, G.: Monodromy for the hypergeometric function ${}_nF_{n-1}$, Invent. Math. \textbf{95}, 325--354 (1989). 

\bibitem[BV19]{BV} Beukers, F., Vlasenko, M.: Dwork crystals {I}, \url{https://arxiv.org/abs/1903.11155} (2019). 

\bibitem[BKSZ22]{BKSZ}
B\"onisch, Kilian, Klemm, Albrecht, Scheidegger, Emanuel, Zagier, Don: D-brane masses at special fibres of hypergeometric families of Calabi-Yau threefolds, modular forms, and periods, \url{https://arxiv.org/abs/2203.09426} (2022).

\bibitem[COEvS20]{COES}
Candelas, P., de la Ossa, X., Elmi, M., van Straten, D.: A one parameter family of Calabi-Yau manifolds with attractor points of rank two. J. High Energ. Phys. 2020, 202 (2020).

\bibitem[CDRV00]{CORV}
Candelas, P., de la Ossa, X., Rodr\'{i}guez Villegas, F.: Calabi--Yau manifolds over finite fields, {I}. \url{https://arxiv.org/abs/hep-th/0012233} (2000).  

\bibitem[CDRV01]{CORV2} 
Candelas, P., de la Ossa, X., Rodr\'{i}guez Villegas, F.: Calabi--Yau manifolds over finite fields, {II}.  In: Calabi--Yau varieties and mirror symmetry.  Fields Inst. Commun., vol. 38, 121-157.  Amer. Math. Soc., Providence (2003).

\bibitem[COvS20]{COS}
Candelas, P., de la Ossa, X., van Straten, D.: Local Zeta Functions From Calabi-Yau Differential Equations. \url{https://arxiv.org/abs/2104.07816} (2021).

\bibitem[CYY08]{CYY}
Chen, Yao-Han, Yang, Yifan, Yui, Noriko: Monodromy of {P}icard-{F}uchs differential equations for {C}alabi-{Y}au threefolds, With an appendix by Cord Erdenberger, J. Reine Angew. Math. \textbf{616}, 167--203 (2008).

\bibitem[CDLNT16]{14thcase}
Clingher, Adrian, Doran, Charles F., Lewis, Jacob, Novoseltsev, Andrey Y., Thompson, Alan: The 14th case {VHS} via {K}3 fibrations. Recent advances in {H}odge theory, London Math. Soc. Lecture Note Ser. \textbf{427}, 165--227 (2016).

\bibitem[CN79]{Conway}
Conway, J.H., Norton, S.P.: Monstrous Moonshine, Bull. London Math. Soc., \textbf{11}, 308--339 (1979). 

\bibitem[CK99]{CK}
Cox, D.A., Katz, S.: Mirror symmetry and algebraic geometry. Mathematical Surveys and Monographs \textbf{68}, American Mathematical Society (1999). 

\bibitem[Dor00]{Doran}
Doran, C.F.: Picard-Fuchs Uniformization and Modularity
of the Mirror Map, Commun. Math. Phys. 212, 625 -- 647 (2000).

\bibitem[DKSSVW18]{zeta}
Doran, C.F., Kelly, T.L., Salerno, A., Sperber, S., Voight, J., Whitcher, U.: Zeta functions of alternate mirror Calabi-Yau families.  Israel J. Math. \textbf{228}, no. 2, 665--705 (2018).

\bibitem[DKSSVW20]{hypergeometric}
Doran, C.F., Kelly, T.L., Salerno, A., Sperber, S., Voight, J., Whitcher, U.: Hypergeometric decomposition of symmetric K3 quartic pencils. Res. Math. Sci. \textbf{7}, article 7 (2020).

\bibitem[DM07]{DM}
Doran, Charles F. and Morgan, John W.: Algebraic topology of {C}alabi-{Y}au threefolds in toric	varieties, Geom. Topol. \textbf{11}, 597--642 (2007).

\bibitem[Dwo69]{padic} 
Dwork, B., {$p$}-adic cycles, Inst. Hautes \'Etudes Sci. Publ. Math. \textbf{37}, 27--115 (1969).

\bibitem[GP90]{GP}
Greene, B.R., Plesser, M.R.: Duality in {C}alabi-{Y}au moduli space, Nuclear Phys. B \textbf{338}, 1, 15--37 (1990).

\bibitem[Har98]{Harnad} 
Harnad, J.: Picard-Fuchs Equations, Hauptmoduls and Integrable Systems, CRM report (1998).

\bibitem[HLYY18]{HLYY}
Huang, A., Lian, B., Yau, S.-T., Yu, C.: Hasse-Witt matrices, unit roots and period integrals. \url{https://arxiv.org/abs/1801.01189} (2018).


\bibitem[HLZ16]{HLZ}
Huang, A., Lian, B., Zhu, X.: Period integrals and the {R}iemann-{H}ilbert correspondence, J. Differential Geom. \textbf{104}, 2, 325--369 (2016).

\bibitem[Igu58]{igusa} 
Igusa, J.I.: Class Number of a Definite Quaternion with Prime Discriminant, Proc. of the Nat. Acad. Sci. 44, 312--314 (1958). 
    
\bibitem[Kad06]{kadir2} 
Kadir, S.: Arithmetic mirror symmetry for a two-parameter family of {C}alabi-{Y}au manifolds, in Mirror symmetry. {V}, AMS/IP Stud. Adv. Math., \textbf{38}, 35--86, Amer. Math. Soc., Providence, RI (2006).
              
\bibitem[KLMSW13]{KLMSW} 
Karp, D., Lewis, J., Moore, D., Skjorshammer, D., Whitcher, U.: On a family of K3 surfaces with $\mathcal{S}_4$ symmetry, in Arithmetic and geometry of K3 surfaces and Calabi-Yau threefolds, 367--386, 
Fields Inst. Commun. \textbf{67} Springer, New York (2013).             
              
\bibitem[Kat72]{katz}
Katz, N.M.: Algebraic solutions of differential equations ({$p$}-curvature and the {H}odge filtration), Invent. Math. {18}, 1--118 (1972).

\bibitem[Kat73]{katzCongruence}
Katz, N.M.: Une formule de congruence pour la fonction $\zeta$, S.G.A. 7 II, Lecture Notes in Mathematics
\textbf{340}, Springer (1973).

\bibitem[Klo18]{kloosterman}
Kloosterman, R.: Zeta functions of monomial deformations of {D}elsarte hypersurfaces, SIGMA Symmetry Integrability Geom. Methods Appl. \textbf{13} Paper No. 087 (2017).

\bibitem[LY96]{LY}
Lian, B.H., Yau, S.-T.: Mirror maps, modular relations and hypergeometric series {II}. Nuclear Phys. B Proc. Suppl. \textbf{46}, 248--262 (1996).

\bibitem[LTYZ21]{LTYZ}
Long, Ling and Tu, Fang-Ting and Yui, Noriko and Zudilin, Wadim: Supercongruences for rigid hypergeometric {C}alabi-{Y}au threefolds. Adv. Math. \textbf{393}, Paper No. 108058, 49 (2021).

\bibitem[MW16]{MW}
Magyar, C., Whitcher. U.: Strong arithmetic mirror symmetry and toric isogenies. Proceedings of the AMS Special Session on Higher Genus Curves and Fibrations of Higher Genus Curves in Mathematical Physics and Arithmetic Geometry, Contemporary Mathematics, American Mathematical Society (2018).

\bibitem[RV01]{RV}
Rodriguez-Villegas, Fernando: Hypergeometric families of {C}alabi-{Y}au manifolds. Calabi-{Y}au varieties and mirror symmetry ({T}oronto, {ON},	2001), Fields Inst. Commun., \textbf{38}, Amer. Math. Soc., Providence, RI, 223--231 (2003).

\bibitem[S{\etalchar{+}}20]{sage}
\emph{{S}ageMath, the {S}age {M}athematics {S}oftware {S}ystem ({V}ersion
 9.1)}, The Sage Developers \url{https://www.sagemath.org} (2020).
 
 

\bibitem[Sal13a]{salerno}
Salerno, A.: An algorithmic approach to the Dwork family. Women in Numbers 2: Research Directions in Number Theory, Contemporary Mathematics \textbf{606} American Mathematical Society (2013). 

\bibitem[Sal13b]{salerno:counting} 
 Salerno, A.: Counting points over finite fields and hypergeometric functions. Funct. Approx. Comment. Math. 49(1): 137-157 (2013). 

\bibitem[SvS15]{samolstraten}
Samol, K., van Straten, D.: Dwork congruences and reflexive polytopes, Ann. Math. Qu\'{e}. \textbf{39} 2, 185--203 (2015).


\bibitem[SV21]{slinkinvarchenko}
Slinkin, A., Varchenko, A.: Hypergeometric integrals modulo {$p$} and {H}asse-{W}itt matrices, Arnold Math. J. \textbf{7} {267--311} (2021).

\bibitem[VY98]{VY} 
Verrill, H., Yui, N.: Thompson series, and the mirror maps of pencils of {$K3$}	surfaces. The arithmetic and geometry of algebraic cycles ({B}anff,	{AB}, 1998). CRM Proc. Lecture Notes, \textbf{24} 399--432 (2000). Amer. Math. Soc., Providence, RI

\bibitem[Vla18]{vlasenko} 
Vlasenko, M.: Higher Hasse–Witt matrices, Indag. Math. \textbf{29} 1411–1424 (2018).

\bibitem[Wan06]{wan}
Wan, D.: Mirror symmetry for zeta functions, Mirror symmetry. {V}, AMS/IP Stud. Adv. Math. \textbf{38} (2006).


\bibitem[Whi15]{k3expository}
Whitcher, U.: Reflexive polytopes and lattice-polarized {K}3 surfaces. Calabi-{Y}au varieties: arithmetic, geometry and physics, Fields Inst. Monogr. \textbf{34}, 65--79, Fields Inst. Res. Math. Sci., Toronto, ON (2015).

\end{thebibliography}
\end{document}